\documentclass[a4paper, 12pt]{article}
\usepackage{mathrsfs}
\usepackage{amsmath, amsthm, amsfonts, amssymb}
\usepackage{color}
\newtheorem{thm}{Theorem}[section]
\newtheorem{thm*}{Theorem*} %番号をつけない．
\newtheorem{cor}[thm]{Corollary}
\newtheorem{lem}[thm]{Lemma}

\theoremstyle{definition}

\DeclareMathOperator{\esssup}{ess\,sup}
\numberwithin{equation}{section} \theoremstyle{remark}
\newtheorem{rem}{Remark}[section]

\begin{document}

\title{ Some effects of the noise intensity upon non-linear stochastic heat equations on $[0,1]$}
\author{ Bin Xie\footnote{ E-mail:
 bxie@shinshu-u.ac.jp; bxieuniv@outlook.com}\\
{ \small{ Department of Mathematical Sciences, Faculty of
Science, Shinshu University }}\\
%{ \small{$^a$ International Young Researchers Empowerment
%Center, }}\\
{ \small{3-1-1 Asahi, Matsumoto, Nagano
390-8621, Japan } } }
\date{
%First  version: June 18, 2015 @ 8:30, Montreal
}
\maketitle

%%%%%%%%%%%%%%%%%%%%%%%%%%%%%%%%%%%%%%%%%%%%%%%%%%%%%%%%%%%%%%%%
%%%%%%%%%%%%%%%%%%%%%%%%%%%%%%%%%%%%%%%%%%%%%%%%%%%%%%%%%%%%%%%%

\begin{abstract}
Various  effects of the noise intensity upon the solution $u(t,x)$ of the stochastic heat equation 
with Dirichlet boundary conditions on $[0,1]$ are  investigated.  We show that for small noise intensity, the $p$-th moment of $\sup_{x\in [0,1]} |u(t,x)|$ is exponentially stable, however, for  large one, it  grows  at least exponentially. We also prove that the noise excitation of the $p$-th energy of $u(t,x)$ is $4$, as the noise intensity goes to infinity.  We formulate a common method to investigate the lower bounds of the above two  different behaviors for large  noise intensity, which are  hard parts in   \cite{FoJo-14}, \cite{FoNu} and \cite{KhKi-15}.  \end{abstract}

%%%%%%%%%%%%%%%%%%%%%%%%%%%%%%%%%%%%%%%%%%%%%%%%%%%%%%%%%%%%%%%%%%
%%%%%%%%%%%%%%%%%%%%%%%%%%%%%%%%%%%%%%%%%%%%%%%%%%%%%%%%%%%%%%%%%%



%%%%%%%%%%%%%%%%%%%%%%%%%%%%%%%%%%%%%%%%%%%%%%%%%%%%%%%%%%
%%%%   Introduction
%%%%%%%%%%%%%%%%%%%%%%%%%%%%%%%%%%%%%%%%%%%%%%%%%%%%%%%%%%
\section{Introduction and main results}
We are interested in various behaviors of  the following stochastic heat equation relative to $\lambda$:
\begin{align}\label{Spde-1.1}
\begin{cases}
\partial_tu(t, x)&=\frac12 \Delta u(t,x) + \lambda\sigma(u(t,x))\dot{w}(t,x), \ t>0, \ x\in (0, 1),\\
u(0,x)& =u_0(x),\   x\in(0, 1),
\end{cases}
\end{align}
where $\lambda>0$ is a positive number, $\sigma$ is a non-random measurable function defined on $\mathbb{R}$ and $\dot{w}(t,x)$ is a Gaussian space-time noise on $[0, \infty) \times [0, 1]$. Such equation is closely connected to the parabolic Anderson model (as $\sigma(u)=u$, see \cite{CaMo-94 }),  the stochastic Burger's equation\cite{BeGi-97, GyNu-99} and the Kardar-Parisi-Zhang (KPZ) equation \cite{BeGi-97, FuQu-14, Ha-13}. Hence some crucial properties , such as the weak intermittency of the solution, are actively studied, see \cite{ChDa-14}, \cite{FoKh-09}, \cite {KhKi-15-1} and references therein.

 In this paper, we are mainly interested  in \eqref{Spde-1.1} with homogeneous Dirichlet boundary  condition, i.e., $u(t,0)=u(t,1)=0$. Some of our results will also hold for \eqref{Spde-1.1} with  homogeneous Neumann boundary condition $\partial_tu(t,0)=\partial_tu(t,1)=0$ and we will state them in form of remarks.

According to \cite{FoNu} and  \cite{KhKi-15}, the parameter $\lambda>0$  in \eqref{Spde-1.1} will be called the level of noise or noise intensity, which is regarded as the inverse temperature.
The solution $u(t,x)$ can be though of the partition function of a continuous space-time random polymer, see \cite{SeSa-08} for more explanations.
%  The effect of $\lambda$ on the $L^2$-energy of stochastic partial differential equations driven by a space-time  is initially studied in \cite{KhKi-15}.  

In this paper, two kinds of the behaviors of the solution relative to noise intensity $ \lambda$ will be studied. To explain our aims and motivations in detail, let us first introduce some notation and the definition of the solution \eqref{Spde-1.1}.
Let $\{\mathcal{F}_t \}_{t\geq 0}$ denote the filtration generated by the $\{w(t,x); t \geq 0, x\in [0,1]\}$, see \cite{Wal}. 
In this paper, we will always assume the following assumption {\bf
A} is satisfied:\\
$(A.1)$ The initial value $u_0$ is non-random and continuous on $[0, 1]$. Furthermore, we assume that the Lebesgue measure of the set ${\it{supp}}(u_0)\cap  [\gamma, 1-\gamma]$ is strictly positive, 
where ${\it supp}(u_0)$ denotes the support of $u_0$ and  $\gamma \in (0, 1/4)$ is fixed hereafter.\\
$(A.2)$ $\sigma(0)=0$ and $\sigma$ is Lipschitz continuous, that is, there exists $K_U>0$ such that for all $u, v\in \mathbb{R}$,
$$|\sigma(u) -\sigma(v)| \leq K_U|u-v|.$$
 Let us recall the definition of the solution to \eqref{Spde-1.1}. Based on the definition introduced in \cite{Wal}, a random field $\{u_\lambda(t,x); t\geq 0, x\in [0, 1]\}$ is said to be a mild solution of \eqref{Spde-1.1} with the homogeneous Dirichlet boundary condition if it is $\mathcal{F}_t$-adapted and continuous in $(t,x)$, and further it satsifies the following integral equation  with probability one 
\begin{align}\label{eq-2.3-1}
u(t,x)&= \int_0^1 g_D(t,x,y)u_0(y)dy + \int_0^t\int_0^1 g_D(t-s,x,y)\lambda \sigma(u(s,y))w(dsdy)\\
&:= D_{1}(t,x) +D_{2, \lambda}(t,x), \notag
\end{align}
where $g_D(t,x,y)$ denotes the fundamental solution (or heat kernel) of the linear part of the stochastic heat equation \eqref{Spde-1.1} with Dirichlet boundary condition $u(t, 0)=u(t,1)=0$.
Similarly, an $\mathcal{F}_t$-adapted and continuous random field $\{u_\lambda(t,x); t\geq 0, x\in [0, 1]\}$ is said to be a mild solution of \eqref{Spde-1.1} with  homogeneous Neumann boundary condition if \eqref{eq-2.3-1} is satisfied almost surely replaced $g_D(t,x,y)$ by the Neumann kernel $g_N(t,x,y)$. 
For the introduction to stochastic partial differential equations, we also refer the reader to \cite{DaKh-09} for more information.

 Since our topics are closely depending on the noise intensity $\lambda$, we will denote by $u_\lambda(t,x)$ the solution of \eqref{Spde-1.1} with homogeneous Dirichlet boundary condition. Let $p\geq 2$ in this paper and then any real valued measurable function $u$ defined on $[0, 1]$, let $\|u\|_{L^p}$ denote its $L^p$-norm on $[0,1]$. Recalling that for $p=\infty$, 
$\|u\|_{L^\infty}=\esssup_{x\in [0, 1]}|u(x)|$.
 
One of our main aims  is to study the exponential stability of the solution for fixed $\lambda$, which is widely studied, because of its importance in applications. One of the important and hard problem for stability is to calculate the Lyapunov exponents. For stochastic parabolic partial differential equations driven by a finite dimensional Gaussian noise, we refer the reader to \cite{Kw-99} and \cite{Xie-08}. However, it seems very hard for \eqref{Spde-1.1}, see \cite{JoKhMu-14}.
Recently, in \cite{FoNu}, the authors proved that if the level of the noise is small, then the $p$-th absolute  moment of $u_\lambda(t,x)$ is exponentially stable, however, for large enough $\lambda$, the $p$-th absolute moment of $u_\lambda(t,x)$ becomes unstable and grows at least exponentially. 
We generalized the main results in \cite{FoNu} in two ways. One is to show the exponential stability of the $p$-th moment of  $\|u(t)\|_{L^\infty}$, instead of $|u_\lambda(t,x)|$, and the other is that an innovative  method to show the lower bound of the growth rate of the solution for large, but fixed $\lambda$, see Theorem \ref{thm-3.1-0608} below.

Our method can also be applied to study  the excitability of the noise as $\lambda \to \infty$ for each $t>0$, which is our second main topic.
The non-linear noise excitability of stochastic heat equations is  initially introduced in \cite{KhKi-15} and restudied in \cite{FoJo-14}, see  also \cite{KhKi-15-1} for other research.  To study this kind of problem, in \cite{KhKi-15} the authors implemented  a projection method to prove that as $\lambda \to \infty,$ the $L^2$-energy of the solution grows at least as $\exp(\kappa_1 \lambda^2)$, and at most as $\exp(\kappa_2 \lambda^4)$. 
But the authors predicted that the lower bound may be $\exp(\kappa_1 \lambda^4)$, instead of $\exp(\kappa_1 \lambda^2)$, that is, the noise excitation may be equal to $4$, same as that for the large number of intermittent complex systems \cite{KhKi-15-1}. To fill this gap, a renewal approach is introduced in \cite{FoJo-14}, which is essentially depends on the short time estimate of the heat kernel.  
They first proved that the noise excitation is $4$ for small time, and then extended it to each fixed $t>0$. However, their method can not be applied to study the larger time behavior as $t\to \infty$ for fixed $\lambda$, which is the first goal of this paper. 
Our main motivation is of this observation and a method is introduced to study both kinds of behaviors above. It seems that our technique can be  applied to (fractional) stochastic heat equations on general bounded domain of $\mathbb{R}^d$ driven by white or colored noises, which is currently being considered.  

From now, we will state our main results. The first is going to show that the $p$-th moment of $\|u_{\lambda}(t)\|_{L^\infty}$ grows exponentially fast at $t\to \infty$. 

For each  $\beta \in \mathbb{R}$ and $p \geq 2$, 
let us  denote by $B_{p, \beta}$ the class of all the $\mathcal{F}_t$-adapted and continuous random field 
$\{u(t,x ), t\geq 0, x\in [0, 1]\}$  satisfying 
\begin{align*}
 \sup_{t\geq 0}
\mathbb{E}[e^{\beta t}\|u_\lambda(t)\|_{L^\infty}^p] <\infty.
\end{align*}
 For each  $u\in B_{p, \beta}$, we set
\begin{align}\label{eq-1.3-0613}
\|u_\lambda\|_{p, \beta}= \Big(\sup_{t\geq 0}
\mathbb{E}[e^{\beta t}\|u_\lambda(t)\|_{L^\infty}^p] \Big)^{\frac1{p}}.
\end{align}
Then it is easy to know that $(B_{p, \beta}, \|u_\lambda\|_{p, \beta})$ is a Banach space.

Let us now formulate the first main result of this paper, which is about the existence and uniqueness of the solution in the Banach spcace  $B_{p, \beta}, \beta<0$ and then gives a upper bound of the growth rate for any $\lambda >0$.
\begin{thm}\label{thm-1.1-0531}
Let $p>2$. Then there exists $\beta_0<0$ such that for any $\beta< \beta_0$,  the equation  \eqref{Spde-1.1} has a unique mild solution $u_\lambda(t,\cdot) \in B_{p, \beta}$.

In particular,  the growth of the solution in time $t$ is at most in a exponential rate in the $p$-moment sense. Precisely speaking, for any $\beta< \beta_0$
\begin{align*}
\limsup_{t\to  \infty}\frac1t\log\mathbb{E}[\|u_\lambda (t)\|_{L^\infty}^p] \leq -\beta.
\end{align*}
\end{thm}
%Form the above theorem, one can know that the following holds:
%\begin{cor}\label{cor-1.1-0531}
%For any  $\lambda>0$ and $p>2$, we have 
%\begin{align}
%\limsup_{t\to  \infty}\frac1t\log\mathbb{E}[\|%u_\lambda(t)\|_{L^\infty}^p] \leq ?.
%\end{align}
%\end{cor}
\begin{rem}
$(i)$
The condition $\sigma(0)=0$ is not required and our result holds for any Lipschitz continuous function $\sigma$.\\ 
$(ii)$
From the proof of this theorem, it is known that the same result holds for \eqref{Spde-1.1} with homogeneous Neumann boundary condition.\\
$(iii)$ In stead of the $L^\infty$-norm of $u_\lambda(t,x)$, the similar behavior of the $p$-th absolute moment of $u_\lambda(t,x)$ is initially investigated in \cite{FoKh-09} for stochastic heat equations on $\mathbb{R}$, and then is restudied in \cite{FoNu} on a bounded domain.  
\end{rem}

 In the next theorem, we are going to  show that if $\lambda$ is small enough, then  the $p$-th  moment of $\|u_\lambda(t)\|_{L^\infty}$ is exponentially stable.
\begin{thm}\label{thm-1.2-0608}
%For any $\lambda>0$ and $p>1$, 
%\begin{align}
%\limsup_{t\to  \infty}\frac1t\log\mathbb{E}[\|u_\lambda(t)\|_{L^\infty}^p] <\infty.
%\end{align}
There exists $\lambda_L$ such that for $\lambda \in (0, \lambda_L)$, 
\begin{align}\label{eq-1.3}
-\infty <\limsup_{t\to  \infty}\frac1t\log\mathbb{E}[|u_\lambda(t,x)|^p] \leq \limsup_{t\to  \infty}\frac1t\log\mathbb{E}[\|u_\lambda(t)\|_{L^\infty}^p] <0.
\end{align}
%\begin{align}
%\limsup_{t\to  \infty}\frac1t\log\mathbb{E}[\|u_\lambda(t)\|_{L^\infty}^p] <0.
%\end{align}
\end{thm}
\begin{rem}
$(i)$ In fact,  we can prove that there exists a $\beta \in (0, (2-\alpha)\pi^2)
$ for some $\alpha\in (2/p, 1)$ such that for all $\lambda \in (0, \lambda_L)$ and all $t \geq 0$, $\mathbb{E}[ \|u(t)\|_{L^\infty}^p ] \leq e^{-\beta t }$, see the proof of this theorem in Section 2.\\
$(ii)$ It is easy to know that for $\lambda <\lambda_L$, the solution $u_\lambda$ is not weakly intermittent. According to \cite{FoKh-09}, we recall that the solution $u_\lambda$ is of weak intermittence  if for any $p\geq 2$, $$\limsup_{t\to  \infty}\frac1t\log\mathbb{E}[|u_\lambda(t,x)|^p]\in (0, \infty).$$ \\
$(iii)$ The lower bound can be easily proved by Jensen's inequality and Theorem 1.1 \cite{FoNu}. As a by-product of the proof of Theorem \ref{thm-3.1-0608} below, we have another proof of it, see Section 3.\\
$(iv)$ We point out that in this paper we will not take $\lambda \to 0$, which is the problem of the large deviations principle for small noises(Freidlin-Wentzell large deviation principle), see \cite{ZhTu-12} and references therein.
\end{rem}

It is now natural to ask what will happen for the  solution $u_\lambda(t,x)$ with large noise intensity $\lambda$  as $t\to \infty$. It is recently studied in  \cite{FoNu}. On the other hand,  as we mentioned in the above, we are also interested in 
the excitation of non-linear noise, see \cite{KhKi-15}, \cite{KhKi-15-1} and \cite{FoJo-14}. The proof for the lower bound  is hard,  as we stated above and different methods are introduced respectively  in \cite{FoJo-14} and \cite{KhKi-15}. It seems that there is no relation between the method for the lower bound in \cite{FoNu} for growth rate as $t\to \infty$ and that in 
\cite{FoJo-14} for the noise excitability. However, we believe that the large time behavior for large, but fixed $\lambda$ and the excitation of non-linear noise are essentially same, and  thus there must be a common approach to study both phenomena. As expected,  we can find such common approach, see Theorem \ref{thm-3.1-0608} below, which is our main contribution.  Our approach essentially depends  on the lower bound of the global estimate for $g_D(t,x, y)$, see Lemma \ref{lem-3.1-0608}.

%Recently, M. Foondun and E. Nualart \cite{FoNu}
%considered the lower bound of  the large time behavior of the solution $u_\lambda(t,x)$ for the fixed level of the noise, see Theorem 1.1 \cite{FoNu} by using different method from that used in \cite{KhKi-15} and \cite{FoJo-14}. However, 
%
Before we state our key theorem, let us further impose the next assumption   on the coefficient $\sigma$: 
\\
$(A3)$ For any $u \in \mathbb{R}$, 
$$ K_{L}|u| \leq \sigma |u|,$$
where $K_L>0$ is a constant. It is clear that $K_U \geq K_L$ is required.\\
Then we have the next theorem, which plays a key role in this paper.
\begin{thm}\label{thm-3.1-0608}
If further $(A3)$ is fulfilled, then there exist two constants $\kappa_1>0$ and $\kappa_2>0$ such that  for all $t>0$
\begin{align}\label{eq-3.1-0608}
\inf_{x\in [\gamma, 1-\gamma]}\mathbb{E}[|u_\lambda(t,x)|^2] \geq \kappa_1\exp(-2\pi^2 t+\kappa_2 \lambda^4 K_L^4 t).
\end{align}
\end{thm}

As the first application of  Theorem \ref{thm-3.1-0608}, we will research on the lower bound of the growth rate of the  $p$-th absolute moment of $u_\lambda(t,x)$ with a large noise intensity $\lambda$ as $t \to \infty$.
\begin{thm}\label{thm-1.3-0608}
%(i) There exists $\lambda_L$  such that for all 
%$\lambda \in (0, \lambda_L)$ {and} $ x\in [0,1]$ 
Under the assumptions in Theorem \ref{thm-3.1-0608}, 
 there  exists $\lambda_U> \lambda_L$
such that for all 
$\lambda \in (\lambda_U, \infty)$ and $ x\in [\gamma, 1-\gamma]$ 
\begin{align} \label{eq-1.9-0608}
0 <\liminf_{t\to  \infty}\frac1t\log\mathbb{E}[|u_\lambda(t,x)|^p] \leq \limsup_{t\to  \infty}\frac1t\log\mathbb{E}[\|u_\lambda(t)\|_{L^\infty}^p] <\infty.
\end{align}
\end{thm}
\begin{rem}
The phenomena in Theorem \ref{thm-1.2-0608} and Theorem \ref{thm-1.3-0608} are peculiar to \eqref{Spde-1.1} with Dirichlet boundary condition, which  are not satisfied for Neumann boundary condition, see \cite{FoNu}.   These display some kind of competition between the noise and dissipativity of the Dirichlet Laplacian. 
%Combining Theorem \ref{thm-1.1} and Theorem 1.1 \cite{FoNu} by M. Foondun and E. Nualart in 2015??, we obtain the following results. However, we will give a new approach to proof the second part, which depends on the lower bound of the heat kernel for all time.

%The above corollary improves the upper bounds of Theorem 1.1 \cite{FuEu}.
\end{rem}

For $p \geq 2$, let us  introduce the $p$-th energy $\mathcal{E}_p(t, \lambda)$ relative to the solution $u_\lambda(t,x)$ of \eqref{Spde-1.1} at time $t>0$ as below:
\begin{align*}
\mathcal{E}_p(t, \lambda) =(\mathbb{E}[\|u_\lambda(t)\|_{L^p}^p])^{\frac{1}{p}}.
\end{align*}
We remark that when $p=2$, $\mathcal{E}_p(t, \lambda)$ is called $L^2$-energy  in  \cite{FoJo-14}  \cite{FoNu}, \cite{KhKi-15}, and \cite{KhKi-15-1} and we generalize it to the definition of $p$-th energy.
\begin{cor}\label{cor-1.5-0611}
Suppose the assumptions in Theorem \ref{thm-1.3-0608} are fulfilled.
Let $\lambda_L$ and $\lambda_U$ be the same constants as that appeared in Theorem \ref{thm-1.2-0608} and Theorem \ref{thm-1.3-0608}  respectively. Then for $\lambda < \lambda_L$,
\begin{align*}
-\infty <\limsup_{t\to  \infty}\frac1t\log\mathcal{E}_p(t, \lambda)  <0,
\end{align*}
and for $\lambda >\lambda_U$
\begin{align} \label{eq-1.9-0608}
0 <\liminf_{t\to  \infty}\frac1t\log\mathcal{E}_p(t, \lambda)<\infty.
\end{align}
\end{cor}

%%% Added at June 10, 2015

Let us now turn to study the non-linear noise excitability of  the stochastic heat equation \eqref{Spde-1.1} by letting thenoise intensity $\lambda$ go to infinity for each $t>0$ as another application of Theorem \ref{thm-3.1-0608}.  
\begin{thm} \label{thm-1.6-0610}
Under the assumptions in Theorem \ref{thm-3.1-0608}, for all $t >0$, there  is constant $c_p>0$ such that 
\begin{align*}
c_p  K_{L}^4 t\leq &  \liminf_{\lambda \to \infty} \lambda^{-4} K_L\log\Big(\inf_{x\in [\gamma, 1-\gamma]}\mathbb{E}[|u_\lambda(t,x)|^p] \Big)\\
 \leq &
\limsup_{\lambda \to \infty} \lambda^{-4} \log\Big( \sup_{x\in [0,1]}\mathbb{E}[|u_\lambda(t,x)|^p] \Big) \leq c_p^{-1} K_U^4 t \notag
\end{align*}
\end{thm}

Then the following theorem exhibits the quantitative behavior of the noise excitability for $p$-th energy.
\begin{cor} \label{cor-1.7-0613}
Under the assumptions in Theorem \ref{thm-1.6-0610},  for all $t >0$, there  is constant $c_p$ such that 
\begin{align*}
c_p  K_{L}^4 t\leq  \liminf_{\lambda \to \infty} \lambda^{-4} K_L\log\mathcal{E}_{p} (t, \lambda) \leq 
\limsup_{\lambda \to \infty} \lambda^{-4} \log\mathcal{E}_{p} (t, \lambda)  \leq c_p^{-1} K_U^4 t.
\end{align*}
\end{cor}

\begin{rem}
$(i)$ The results in Theorem \ref{thm-1.6-0610} and Corollary \ref{cor-1.7-0613} still hold for \eqref{Spde-1.1} with Neumann boundary condition, see \cite{FoJo-14} and \cite{KhKi-15}.\\ 
$(ii)$ Analogously to \cite{KhKi-15-1}, let us introduce the noise excitation index of the solution $u_\lambda(t,x)$ relative to $p$-th energy $\mathcal{E}_{p} (t, \lambda)$.
If for each $t>0$, $$
\lim_{\lambda \to \infty} \frac{\log\log \mathcal{E}_p(t, \lambda)}{\log \lambda}
$$ exists, then its limit denoted by $e_p(t)$ is called the  noise excitation index of $p$-th energy. If furthermore,  $e_p(t)$ does not depend on $t$, then the common value denoted by $e_p$ is called the index of nonlinear noise excitation of the $p$-th energy $\mathcal{E}_{p, t} (t, \lambda)$. It is clear from the above theorem that $e_p=4$ which is independent of $p, p\in [2, \infty)$. The definition of noise excitation index is initially introduced by D. Khoshnevisan, K. Kim in \cite{KhKi-15-1} and we refer the reader to this paper for its significance.

Let us recall that for $p=2$,  M. Foondun and M. Joseph\cite{FoJo-14} proved that $e_2=4$, which improved a  result in Theorem \cite{KhKi-15} by using a renewal approach based on the short time estimate of the heat kernel. 
\end{rem}
  
    The paper is organized as follows:  In Section 2, we  give  proofs of Theorem \ref{thm-1.1-0531} and Theorem \ref{thm-1.2-0608} based on some lemmas. In Section 3, we  first state a lower bound of the global time estimate for the heat kernel and then prove our important result, Theorem \ref{thm-3.1-0608}. 
Proofs of Theorem \ref{thm-1.3-0608}  and Theorem \ref{thm-1.6-0610} and their corollaries are formulated in Section 4 and Section 5 respectively.  In the end, for the reader's convenience, we write down a version of Garsia-Rodemich-Rumsey theorem in Appendix, which is cited in Section 2.
%%%%%%%%%%%%%%%%%%%%%%%%%%%%%%%%%%%%%%%%%%%%%%%%%%%%%%%%%%%%%
%%%%
%%%%%%%%%%%%%%%%%%%%%%%%%%%%%%%%%%%%%%%%%%%%%%%%%%%%%%%%%%%%%
\section{Proof of Theorem \ref{thm-1.1-0531} and Theorem \ref{thm-1.2-0608} }
%%%%%%%%%%%%%%%%%%%%%%%%%%%%%%%%%%%%%%%%%%%%%%%%%%%%%%%%%%%%%
%%%  %%%%%%%%%%%%%%%%%%%%%%%%%%%%%%%%%%%%%%%%%%%%%%%%%%%%%%%%%%
%% The noise
%%%%%%%%%%%%%%%%%%%%%%%%%%%%%%%%%%%%%%%%%%%%%%%%%%%%%%%%%%
Since we consider the mild solution $u_\lambda(t,x)$ of \eqref{Spde-1.1} with homegeneous Dirichlet boundary condition, most of our calculations depend on various estimates of $g_D(t,x,y)$. We will  recall necessary  properties of $g_D(t,x, y)$ when they are required. Firstly, by the spectral theory, it is well-known that 
\begin{align}\label{eq-2.1}
g_D(t,x,y) =2\sum_{n=1}^\infty e^{-(n\pi)^2 t}\sin(n\pi x)\sin(n\pi y),\ x, y\in [0, 1].
\end{align}
It is also easy to know that  $t>0$ and $x, y \in [0, 1]$, 
\begin{align}\label{eq-2.2}
0\leq g_D(t,x,y) \leq g(t,x,y)
\end{align}
where $g(t,x,y)$ denotes the transition probability density of some one-dimensional standard Brownian motion.

Different from the study of the long time behaviour of $\mathbb{E}[|u_\lambda(t,x)|^p]$ in \cite{FoNu}, to study that of $\mathbb{E}[\|u(t)\|_{L^\infty}^p]$, the estimate of the derivative of $g_D(t,x,y)$ is vital, that is
\begin{align}\label{eq-2.2-1}
| \partial_xg_D(t,x, y)| \leq K_1 t^{-1}e^{-K_2\frac{(x-y)^2}{t}}. 
\end{align}
where $K_1$ and $ K_2$ are two generic positive constants.

Before we state the proof of Theorem \ref{thm-1.1-0531}, we will  give some lemmas.
Let us first formulate the famous Kolmogorov's regularity theorem with its brief proof for  the reader's convenience and our purpose.
\begin{lem}(Kolmogorov's regularity theorem)\label{lem-2.2}
Let $\{u(x)\}_{x\in [0, 1]}$ be a real valued stochastic process.  If there exist $p\geq 1$ and positive constants $K, \delta$ such that 
\begin{align}\label{eq-2.5}
\mathbb{E}[|u(x)-u(y)|^p] \leq K|x-y|^{1+\delta}.
\end{align} 
then we have that for each $\epsilon \in (0, \min\{\delta, 1\})$, 
there exists a positive constant $\kappa$ depending only on $p, \delta, \epsilon$ such that
\begin{align}\label{eq-2.7}
|u(x) - u(y)| \leq \kappa B^{1/p}|x-y|^{\frac{\delta- \epsilon}{p}},
\end{align}
where  $B=B(\epsilon, \delta)$ is the positive random variable defined by 
\begin{align}\label{eq-2.6}
B=\int_0^1\int_0^1\frac{|u(x)-u(y)|^p}{|x-y|^{2+\delta -\epsilon}}dxdy. 
\end{align}
In particular, the stochastic process $\{u(x)\}_{x\in [0, 1]}$ has a $\frac{\delta- \epsilon}{p}$-H\"{o}lder continuous modification.
\end{lem}
\begin{proof}
This lemma is a modification of  Corollary 1.2 \cite{Wal}.  Similarly, we state its proof briefly using the celebrated  analytic inequality introduced by Garsia, Rodemich and Rumsey, see the original research paper \cite{GRR}, Theorem 1.1 \cite{Wal} or  Theorem \ref{GRR} in Appendix.
Let us  consider $\Phi(x)=|x|^p,\  x\in \mathbb{R}$ and  $\phi(x)=|x|^{\frac{2+\delta -\epsilon}{p}},\  x\in [0,1]$. It is clear that the two functions  $\Phi(x)$ and $\phi(x)$ satisfy the conditions in  Theorem \ref{GRR}.
We can now rewrite the random variable $B$ defined by \eqref{eq-2.6} as 
\begin{align*}
B=\int_0^1\int_0^1 \Phi\Big(\frac{u(x)-u(y)}{\phi(|x-y|) } \Big)dxdy.
\end{align*}
Taking the expectation of $B$, and combining with the condition \eqref{eq-2.5} on $u(x)$, we can easily arrive at 
\begin{align}\label{eq-2.14}
\mathbb{E}[B] &=K\int_0^1\int_0^1{|x-y|^{-1 +\epsilon}}dxdy 
= \frac{2K}{\epsilon(1+\epsilon)}<\infty, %\notag
\end{align}
where  $\epsilon \in (0, \min\{\delta, 1\} )$ has been used.
Thus, we can apply the Garsia-Rodemich-Rumsey Theorem, see Theorem \ref{GRR}.  Using Theorem \ref{GRR} and the  integration by parts, we obtain that 
\begin{align*}
|u(x)-u(y)| &\leq 8\int_0^{|x-y|} B^{\frac{1}{p}}u^{-\frac{2}{p}}d(u^{\frac{2+\delta -\epsilon}{p}})  \\
& 
=8B^{\frac1p} [|x-y|^{\frac{\delta -\epsilon }{p}} + \frac2p \int_0^{|x-y|} u^{\frac{\delta -\epsilon}{p}-1}du] \\
& =8B^{\frac1p}(1+(\delta -\epsilon)^{-1} )|x-y| ^{\frac{\delta -\epsilon }{p}}. 
\end{align*}
Taking $\kappa= 8(1+(\delta -\epsilon)^{-1} )$,  the estimate of \eqref{eq-2.7} is proved. Finally, the existence of $\frac{\delta- \epsilon}{p}$-H\"{o}lder continuous modification for the stochastic process $\{u(x)\}_{x\in [0, 1]}$ is obvious from  \eqref{eq-2.6}. Consequently, our proof is completed.
\end{proof}

\begin{rem}
As we said in the above proof, we mainly imitated the approach for Corollary 1.2\cite{Wal}. So it may be considered that our proof is insignificant. In fact, the proofs for Theorem \ref{thm-1.1-0531} and Theorem \ref{thm-1.2-0608} essentially depends on the estimates of $K$ and $B$ relative to the solution of $u_\lambda(t,x)$. However, we can not know the concrete form of the right hand side of \eqref{eq-2.7}, if we do not read the proof of  Corollary 1.2\cite{Wal} carefully. On the other hand, although the similar result to \eqref{eq-2.7} in Corollary 1.2\cite{Wal} is  better that \eqref{eq-2.7} , our result is concise and easy to be applied.
\end{rem}

%%%%%%%%
%%
\begin{lem}\label{lem-2.1-0531}
Assume $\alpha \in (0, 1)$. Then there exists a constant $C>0$ depending on $\alpha$ such that for any $\beta<0$
\begin{align*}%\label{eq-2.10-0601}
\sup_{t\geq 0,x\in [0, 1] } \int_0^t  e^{\beta s}s^{-\alpha} \int_0^1|g_D(s, x, y)|^{2-\alpha}dyds \leq C |\beta|^{\frac{\alpha-1}{2}}.
\end{align*} 
\end{lem}
\begin{proof}
This proof is very easy by noting t that the $g_D(t,x,y) \leq g(t,x,y)$. 
In fact, by \eqref{eq-2.2} and the conditions  $\alpha \in(0, 1)$ and $\beta<0$, we have
\begin{align*}%\label{eq-2.6-1}
& \int_0^t e^{\beta s}s^{-\alpha} \int_0^1 |g_D(s, x, y)|^{2-\alpha}dy ds\\
\leq&
\int_0^t e^{\beta s}s^{-\alpha} \int_{\mathbb{R}} |g(s, x, y)|^{2-\alpha}dy  ds\notag\\
=& (2\pi)^{\frac{\alpha -1}{2}}(2-\alpha)^{-\frac12}
   \int_0^t e^{\beta s}s^{-\frac{1+\alpha}{2}}ds \notag\\ 
\leq & (2\pi)^{\frac{\alpha -1}{2}} 
         (2-\alpha)^{-\frac12} |\beta|^{\frac{\alpha-1}{2}}
   \int_0^{\infty} e^{-s}s^{-\frac{1+\alpha}{2}}ds \notag \\
= & (2\pi)^{\frac{\alpha -1}{2}} 
         (2-\alpha)^{-\frac12} \Gamma( \frac{1-\alpha}{2}) |\beta|^{\frac{\alpha-1}{2}}, \notag
  % \notag
%\leq & (2\pi)^{\frac{\alpha -1}{2}}       (2-\alpha)^{-\frac12} e^{2\pi^2}\beta^{\frac{\alpha-1}{2}}\int_0^{2\pi^2} s^{-\frac{1+\alpha}{2}}ds  \notag
\end{align*}
where $\Gamma$ denotes the Gamma function. Thus, we can conclude the proof.
\end{proof}

For any $u\in B_{p, \beta}$, let us define the mapping $Su(t,x)$   by  the  stochastic convolution
\begin{align*}
Su(t,x) =\int_{0}^t\int_0^t g_D(t-s,x, y)u(s, y)w(dsdy).
\end{align*}
In the following two lemmas, we will show that  $S$ maps $B_{p, \beta}$ to itself and it is contractive for some $\beta$, respectively.

\begin{lem}\label{lem-2.3-0601}
Assume $\beta<0$ and $p>2$. Then for each $\alpha\in (2/p,1)$, there exists a constant 
$C=C(p, \alpha)>0$ independent of $\beta$ such for all $u\in B_{p, \beta}$, 
\begin{align}\label{eq-2.12-0601}
\|Su\|_{p,\beta}^p \leq   C \|u\|_{p, \beta}^p (|\beta|^{\frac{-p(1-\alpha)}{4}}+ |\beta|^{-p/4})
\end{align}
\end{lem}
\begin{proof}
The key point of this proof is the application of Lemma \ref{lem-2.2}. 
By Burkh\"older's inequality, for any $x, y\in [0, 1]$ and $t\geq 0$
\begin{align*}
& \mathbb{E} [|Su(t,x)(t,x) -S(t,y)|^p]\\
& \leq \kappa^p(p) \mathbb{E}  \left[ \left(\int_0^t \int_{0}^1 (g(t-s, x, z)-g(t-s,y, z))^2 u^2(s,z))dzds \right)^{p/2} \right]\notag
\end{align*}
where $\kappa(p)$ denotes the optimal constant in Burkh\"{o}lder's $L^p(\Omega)$-inequality for continuous square-integrable martingales, see \cite{Dav-76}.
Using the continuous version of Minkowski's inequality, see, for example, Theorem 6.2.7 \cite{Str-11}, and the above estimate, we have that 
 \begin{align*}
& \mathbb{E}[|Su(t,x) -Su(t,y)|^p]^{2/p}\\
& \leq \kappa^p(p) \left( \int_0^t \int_{0}^1 (g(t-s, x, z)-g(t-s,y, z))^2  \mathbb{E}[ |u(s,z)|^p]^{2/p}dzds \right)^{p/2}. \notag
\end{align*}
Applying the mean value theorem and \eqref{eq-2.2-1}, we obtain that
for each $\alpha \in (0,1)$ 
\begin{align}\label{eq-2.18-0612}
& \mathbb{E}[|Su(t,x) -Su(t,y)|^p]\\
& \leq \kappa^p(p)  \left(\int_0^t \int_{0}^1 \Big|\int_0^1 \partial_x g_D(t-s,x+\theta(x-y), z)d\theta \Big|^\alpha\right. \notag \\
&\qquad  \times  |g(t-s, x, z)-g(t-s, y, z)|^{2-\alpha}    \mathbb{E}[ |u(s,z)|^p]^{2/p}dzds \Big)^{p/2} |x-y|^{\alpha p/2} \notag \\
& := K_1(t)|x-y|^{\alpha p/2} \notag
\end{align}
From now, we will assume that $\alpha p>2$. We point out that because of $p>2$,  it is possible for us to choose $\alpha \in (\frac2p, 1)$ such that  $\alpha p>2$.
Applying Lemma \ref{lem-2.2} to $Su(t,x)$, we deduce from 
\eqref{eq-2.18-0612} that for all $t\geq 0$ and all $x, y \in [0, 1]$
\begin{align}\label{eq-2.13-0612}
|Su(t,x) -Su(t,y)| \leq  \kappa B(t)^{\frac{1}{p}}
|x-y|^{\frac{\alpha}{2} -\frac{1+\epsilon}{p}}, 
\end{align}
where $\kappa$ is same as that in \eqref{eq-2.7} in Lemma \ref{lem-2.2},
\begin{align*}%\label{eq-2.20-0601}
B(t)=\int_0^1\int_0^1\frac{|Su(t, x)-Su(t, y)|^p}{|x-y|^{1+\alpha p/2 -\epsilon}}dxdy
\end{align*}
and $\epsilon\in (0, \min\{\alpha p/2-1, 1\}).$
% satisfy the same conditions in  Lemma \ref{lem-2.2}.
 It is valuable to point out that the constant $\kappa$ appeared in  \eqref{eq-2.13-0612} does not depend on time $t$, which is very important for our goal.
Let $y=y_0\in [0, 1]$ be fixed. Then from  \eqref{eq-2.13-0612} and noting that $\frac{\alpha}{2} - \frac{1+ \epsilon}{p}>0$, we can deduce that for any $t \geq 0$ and $x\in [0, 1]$
\begin{align*}
|Su(t,x)| \leq  \kappa B(t)^{1/p}
|x|^{\frac{\alpha}{2} - \frac{1+\epsilon}{p}} + |Su(t, y_0)|, 
\end{align*}
which implies that  for any $t>0$
\begin{align*}
\sup_{x\in [0, 1]} |Su(t,x)| \leq  \kappa B(t)^{\frac{1}{p}} + |Su(t, y_0)|,
\end{align*}
%and then by \eqref{eq-2.20} and \eqref{eq-2.19} again,
Let us now take the $p$-th moments of the both sides of the above inequality and use the inequality $|a+b|^p \leq 2^{p-1}(|a|^p +|b|^p),\ a,b \in \mathbb{R}$, we obtain that 
\begin{align}\label{eq-2.16-0601}
\mathbb{E}\Big[\sup_{x\in [0, 1]} |Su(t,x)|^p \Big]\leq  2^{p-1}(\kappa^p \mathbb{E}[B(t)] + \mathbb{E}[|Su(t, y_0)|^p]).
\end{align}
By \eqref{eq-2.18-0612} and analogously to  \eqref{eq-2.14}, we know that  there is a positive constant $C_1$ depending  on $\epsilon$ such that  for all $t \geq 0$,
\begin{align}\label{eq-2.14-0615}
 \mathbb{E}[B(t)] \leq C_1K_1(t).
\end{align}
Recalling that $K_1(t)$ is defined in \eqref{eq-2.18-0612}.
Let us now pay attention to the estimate of $K_1(t)$. %based on Lemma \ref{lem-2.1}.
%From the assumptions on $\sigma$, 
From \eqref{eq-2.2-1}, it follows that 
$$
\Big|\int_0^1 \partial_x g_D(t-s,x+\theta(x-y), z)d\theta \Big| \leq  K_1t^{-1}.
$$
Thus, by the definition of $K_1(t)$, see \eqref{eq-2.18-0612} above and  Lemma \ref{lem-2.1-0531}, it follows that   for all $t \geq 0$ and any $x, y \in [0, 1]$
\begin{align}\label{eq-2.18-0601}
K_1(t)
\leq & \kappa^p(p) C_2  \Big( \int_0^t \int_{0}^1 (t-s)^\alpha |g(t-s, x,z)-g(t-s,y, z)|^{2-\alpha}   \\
&\qquad  \times  \mathbb{E}[ \|u(s)\|_{L^\infty}^p]^{2/p}dzds\Big)^{p/2} \notag \\
& =  \kappa^p(p) C_2 \Big( \int_0^t \int_{0}^1 (t-s)^\alpha |g(t-s, x,z)-g(t-s,y, z)|^{2-\alpha}  \notag \\
&\qquad  \times\big(e^{\beta(t-s)}  e^{-\beta(t-s)} \mathbb{E}[ \|u(s)\|_{L^\infty}^p]\big)^{2/p} dzds\Big)^{p/2} \notag \\
& \leq   \kappa^p(p) C_2e^{-\beta t}  \|u\|_{p, \beta}^p \Big( \int_0^t \int_{0}^1e^{\frac{2\beta s}{p}}  s^\alpha   |g(s, x,z)-g(s,y, z)|^{2-\alpha} dzds\Big)^{p/2} \notag \\
&\leq \kappa^p(p) C_3 \ e^{-\beta t}  \|u\|_{p, \beta}^p |\beta|^{\frac{p(\alpha-1)}{4}}, \notag
\end{align}
where $u \in B_{p, \beta}$ has  been used for the third line, and $C_2, \ C_3$ are two generic constants depending on $p$ and $\alpha$. 

We can more easily give the  estimate of the term $\mathbb{E}[Su(t, y_0)|^p])$ in \eqref{eq-2.16-0601}.
%It is similar to \eqref{eq-2.18-0601} and Lemma \ref{lem-2.1-0531}, 
In fact,  similarly to \eqref{eq-2.18-0601}, by using Minkowski's inequality, Burkh\"older's inequality and  the semigroup property $\int_0^1 g_D^2(s, y_0, y)dy =g_D(2s,y_0, y_0)$, we deduce that 
\begin{align}\label{eq-2.19-0601}
\mathbb{E}[Su(t, y_0)|^p]) &\leq 
\kappa^p(p) \left(\int_0^t \int_{0}^1 g_D^2(t-s,y_0, y)
 \mathbb{E}[|u_\lambda(s,y)|^p]^{2/p}dyds \right)^{p/2}  \\
& \leq  \kappa^p(p)  e^{-\beta t} \left(\int_0^t \int_{0}^1 g_D^2(t-s, y_0, y)e^{\frac{2\beta(t-s)}{p}}
\mathbb{E}[e^{\beta s}\|u(s)\|^p]^{2/p}dyds \right)^{p/2}  \notag \\
& \leq  \kappa^p(p) e^{-\beta t} \|u\|_{p, \beta}^p  \left(\int_0^t  g_D(2s, y_0, y_0)e^{\beta s}ds \right)^{p/2} 
\notag \\
& \leq  \kappa^p(p) e^{-\beta t} \|u\|_{p, \beta}^p  \left(\int_0^\infty  (4\pi s)^{-1/2}e^{\beta  s }ds \right)^{p/2}  \notag
\\
& \leq  \kappa^p(p)  e^{-\beta t} \|u\|_{p, \beta}^p  
|\beta|^{-p/4}\left((4\pi )^{-1/2} \Gamma(\frac12) \right)^{p/2} \notag
\\
& =  2^{-p/2}\kappa^p(p)  e^{-\beta t} \|u\|_{p, \beta}^p  
|\beta|^{-p/4}. \notag
\end{align}

 Consequently, plugging \eqref{eq-2.18-0601} and \eqref{eq-2.19-0601} into \eqref{eq-2.16-0601}, we deduce that there exists a constant $C_4= C_4(p, \alpha)>0$ such that 
\begin{align*}%\label{eq-2.28}
\mathbb{E}\Big[\sup_{x\in [0, 1]} |Su(t,x)|^p \Big]\leq C_4 e^{-\beta t} \|u\|_{p, \beta}^p (|\beta|^{\frac{-p(1-\alpha)}{4}}+ |\beta|^{-p/4}).
\end{align*}
Multiplying both sides of the above inequality  by $e^{\beta t}$and  then taking the supremum for $t\geq 0$, we go to
\begin{align*}%\label{eq-2.28}
\|Su\|_{p, \beta}^p\leq C_4  \|u\|_{p, \beta}^p (|\beta|^{\frac{-p(1-\alpha)}{4}}+ |\beta|^{-p/4}),
\end{align*}
which completes the proof of  \eqref{eq-2.12-0601}. 
\end{proof}

\begin{lem}\label{lem-2.4-0601}
Assume $\beta<0$ and $p>2$. Then for each $\alpha\in (2/p,1)$, there exists a constant 
$C=C(p, \alpha)>0$ independent of $\beta$ such for any $u, v \in B_{p, \beta}$, 
\begin{align*}%\label{eq-2.12-0613}
\|Su-Sv\|_{p,\beta}^p \leq   C \|u-v\|_{p, \beta}^p (|\beta|^{\frac{-p(1-\alpha)}{4}}+ |\beta|^{-p/4})
\end{align*}
\end{lem}
\begin{proof}
The proof is essentially same as that of Lemma \ref{lem-2.3-0601}. So we will only write down the different parts and leave the details to the reader.
Form the definition of $Su$, it follows that 
\begin{align*}
Su(t, x)-Sv(t,x)= \int_0^t\int_0^1 g_D(t-s,x,z) \big(u(s,z) - v(s,z)\big)w(ds,dy).
\end{align*}
Then, similarly to  \eqref{eq-2.18-0601}, 
\begin{align*}%\label{eq-2.22-0601}
& \mathbb{E}[|Su(t,x)-Sv(t,x)- (Su(t,y)-Sv(t,y))|^p]\\
& \leq \kappa^p(p) |x-y|^{\alpha p/2} \left(\int_0^t \int_{0}^1 \Big|\int_0^1 \partial_x g_D(t-s,x+\theta(x-y), z)d\theta \Big|^\alpha\right. \notag \\
&\qquad  \times  |g(t-s, x, z)-g(t-s, y, z)|^{2-\alpha}    \mathbb{E}[ |u(s,z) - v(s,z) |^p]^{2/p}dzds \Big)^{p/2} \notag \\
& := K_2(t)|x-y|^{\alpha p/2}.\notag
\end{align*}
Noting that $\alpha \in (\frac{2}{p}, 1)$ and using Lemma \ref{lem-2.2}, there exists a constant $C_1=C_1(p, \alpha)>0$ such that
\begin{align}\label{eq-2.23-0601}
\mathbb{E}\Big[\sup_{x\in [0, 1]} |Su(t,x)- Sv(t,x)|^p \Big]\leq  C_1 \big(K_2(t)+ \mathbb{E}[|Su(t, y_0) -  Sv(t, y_0)|^p] \big), y_0\in [0,1].
\end{align}
On the analogy of \eqref{eq-2.18-0601} and \eqref{eq-2.19-0601}, we can deduce that for constants $C_2$ and $C_3$ depending only on $p$ and $\alpha$,
$t \geq 0$ and any $x, y \in [0, 1]$
\begin{align}\label{eq-2.24-0601}
K_2(t) \leq C_2  e^{-\beta t}  \|u-v\|_{p, \beta}^p |{\beta}|^{\frac{p(\alpha-1)}{4}}, 
\end{align}
and
\begin{align}\label{eq-2.25-0601}
\mathbb{E}[Su(t, y_0)-Sv(t, y_0)|^p]
 \leq C_3 e^{-\beta t} \|u-v\|_{p, \beta}^p  |\beta|^{-p/4}.
\end{align}
Thus, we can easily obtain our result, by plugging \eqref{eq-2.24-0601} and \eqref{eq-2.25-0601} into  
\eqref{eq-2.23-0601}.
\end{proof}

%%% Proof for Theorem \ref{thm-1.1-0531}
Let us now  begin to formulate the proof of Theorem  \ref{thm-1.1-0531} by using  Lemma \ref{lem-2.3-0601} and Lemma \ref{lem-2.4-0601}. 

\begin{proof}{(\it Proof of Theorem \ref{thm-1.1-0531})}
Without loss of the generality, let us suppose that $\lambda =1$ in this part.
Since $u_0$ is continuous on $[0, 1]$ and non-random,
we see that for any $\beta <0$
\begin{align*}
 \|D_{1}(t,x)\|_{p, \beta}^p=& \sup_{t\geq 0}e^{\beta t}\left|\int_0^1g_D(t ,x, y)u_0(y) dy \right|^p \\
 & \leq \|u_0\|_{L^\infty}^p\sup_{t\geq 0}e^{\beta t}
 \left|\int_0^1g(t ,x, y)dy \right|^p \\
& \leq \|u_0\|_{L^\infty}^p,
\end{align*}
which implies  for any $\beta<0$, $D_{1}(t,x) \in B_{p,\beta}$; recalling  that $D_{1}(t,x)$ is defined in \eqref{eq-2.3-1}.

Consider the following operator $T$ on $B_{p,\beta}$ by  
$$Tu(t,x) =D_{1}(t,x) +S(\sigma(u(t,x)).$$
By $(A2)$, it is known that $|\sigma(u)| \leq K_U|u|$. Hence, by Lemma \ref{lem-2.3-0601}, for any $u(t,x )\in B_{p,\beta}$,
$$ \|Tu\|_{p,\beta}^p \leq   \|u_0\|_{L^\infty}^p+  C K_U^p\|u\|_{p, \beta}^p (|\beta|^{\frac{-p(1-\alpha)}{4}}+ |\beta|^{-p/4})<\infty,$$
which $\alpha\in (2/p,1)$ and $C$ is the constant in Lemma \ref{lem-2.3-0601}.
Thus,  we have that 
$T$ maps $B_{p, \beta}<0$ into itself for any $\beta$. 

Analogously, from Lemma \ref{lem-2.4-0601}, it follows that  for each $\alpha\in (2/p,1)$ and  for any $u, v \in B_{p, \beta}$, 
\begin{align}\label{eq-2.20-0618}
\|Tu-Tv\|_{p,\beta}^p = K_U^p\|Su-Sv\|_{p,\beta}^p  \leq CK_U^p \|u-v\|_{p, \beta}^p (|\beta|^{\frac{-p(1-\alpha)}{4}}+ |\beta|^{-p/4}),
\end{align}
where $C$ is the constant in Lemma \ref{lem-2.4-0601}. Recall that $C$ depends on $p, \alpha$, which  is independent of $\beta$.
Since \eqref{eq-2.20-0618} is satisfied for any $\beta<0$, and noting that $\alpha \in (2/p, 1)$, we choose a $\beta_0<0$ such that for any $\beta< \beta_0$, 
$$CK_U^p(|\beta|^{\frac{-p(1-\alpha)}{4}}+ |\beta|^{-p/4})<1,$$
which implies that for any $\beta< \beta_0<0$, the operator
$T$ on the Banach space $B_{p, \beta}$ is contractive. Consequently, the existence and uniqueness of the solution $u$ in $B_{p, \beta}$ is proved.  

The second part is a direct conclusion of the above proof, by noting that for any $\beta < \beta_0$, $\|Tu\|_{p,\beta}<\infty.$ 
\end{proof}

%%%%%%%%%%%%%%%%%%%%%%%%%%
%%% 
%\section{Proof of Theorem \ref{thm-1.2-0608}}
From now, we are going to give the proof of Theorem \ref{thm-1.2-0608}. 
To do it, we will state another property of $g_D(t,x,y)$. From the concrete form $g_D(t,x,y)$, see \eqref{eq-2.1}, it is easy to see that there exists a constant $K_3>0$ such that for any $t\geq 1$,
\begin{align}\label{eq-2.3}
g_D(t,x,y) \leq K_3 e^{-\pi^2 t}.
\end{align}
The next lemma is required.
\begin{lem}\label{lem-2.1}
Assume $\alpha \in (0, 1)$ and $\beta \in (0, (2-\alpha)\pi^2 )$. Then there exists a constant $C>0$ depending on $\alpha, $ such that for any $t \geq 0$ and $x\in [0, 1]$
\begin{align*}%\label{eq-2.5-1}
\int_0^t e^{\beta s}s^{-\alpha} \int_0^1|g_D(s, x, y)|^{2-\alpha}dyds \leq C \Big(\beta^{\frac{\alpha-1}{2}} + \frac{1}{(2-\alpha)\pi^2 -\beta} \Big).
\end{align*} 
\end{lem}
\begin{proof}
This proof can be easily completed by using the properties \eqref{eq-2.2} and \eqref{eq-2.3} of  $g_D(t,x,y)$.
In fact, noting that \eqref{eq-2.2}, we have
\begin{align}\label{eq-2.6-1}
& \int_0^1 e^{\beta s}s^{-\alpha} \int_0^1 |g_D(s, x, y)|^{2-\alpha}dy ds\\
\leq&
\int_0^1 e^{\beta s}s^{-\alpha} \int_{\mathbb{R}} |g(s, x, y)|^{2-\alpha}dy  ds\notag\\
=& (2\pi)^{\frac{\alpha -1}{2}}(2-\alpha)^{-\frac12}
   \int_0^1 e^{\beta s}s^{-\frac{1+\alpha}{2}}ds \notag\\ 
\leq & (2\pi)^{\frac{\alpha -1}{2}} 
         (2-\alpha)^{-\frac12} \beta^{\frac{\alpha-1}{2}}
   \int_0^{2\pi^2} e^{s}s^{-\frac{1+\alpha}{2}}ds. \notag
  % \notag
%\leq & (2\pi)^{\frac{\alpha -1}{2}}       (2-\alpha)^{-\frac12} e^{2\pi^2}\beta^{\frac{\alpha-1}{2}}\int_0^{2\pi^2} s^{-\frac{1+\alpha}{2}}ds  \notag
\end{align}
Since $\alpha \in (0, 1)$,  the integrand $e^{s}s^{-\frac{1+\alpha}{2}}$ above is integrable on $[0, 2\pi^2]$, from \eqref{eq-2.6-1}, there exists a  constant $C_1>0$ depending only on $\alpha$ such that 
\begin{align}\label{eq-2.6-2}
 \int_0^1 e^{\beta s}s^{-\alpha} \int_0^1 |g_D(s, x, y)|^{2-\alpha}dy ds
\leq C_1 \beta^{\frac{\alpha-1}{2}}.
\end{align}
On the other hand, by \eqref{eq-2.3} and $\beta \in [0, (2-\alpha)\pi^2 )$, we see that 
\begin{align}\label{eq-2.7-1}
& \int_1^\infty e^{\beta s}s^{-\alpha} \int_0^1 |g_D(s, x, y)|^{2-\alpha}dy \\
\leq&
K_3^{2-\alpha}\int_1^\infty e^{\beta s}s^{-\alpha} \int_0^1 e^{(2-\alpha)\pi^2s}dyds 
 \notag \\
= &
K_3^{2-\alpha}\int_1^\infty e^{(\beta- (2-\alpha)\pi^2 ) s}s^{-\alpha}  ds  \notag\\
 \leq &K_3^{2-\alpha} \int_1^\infty e^{(\beta- (2-\alpha)\pi^2) s} ds \notag\\
%= &\kappa_3^{2-\alpha}\int_1^\infty e^{(\beta- (2-\alpha)\pi^2) s}s^{-\alpha}  ds  \notag\\
\leq & C_2 ((2-\alpha)\pi^2 -\beta)^{-1}, \notag
\end{align}
where %$\beta < (2-\alpha)\pi^2$ was used for the last inequality and 
$C_2>0$ is a constant depending  on $\alpha$.  As a consequence of \eqref{eq-2.6-2} and \eqref{eq-2.7-1}, we can complete our proof.
\end{proof}
%\begin{rem}
%When $\alpha =0$, the above lemma is proved in Lemma 2.1 %\cite{FoNu}.
%\end{rem}

In the following, we will formulate the proof of Theorem \ref{thm-1.2-0608}. 

\begin{proof}({\it Proof of  Theorem \ref{thm-1.2-0608} })
By Jensen's inequality and Theorem 1.1 \cite{FoNu},
we easily have $$ 
-\infty <\limsup_{t\to  \infty}\frac1t\log\mathbb{E}[|u_\lambda(t,x)|^p].
$$
Thus, the main task is to give the proof for the upper bound.
Let us assume that $\alpha\in (2/p,1)$ and $\beta \in (0, (2-\alpha)\pi^2 )$ in this part. 
Recalling the definition of $\|u_\lambda\|_{p, \beta}$, see \eqref{eq-1.3-0613}, it is sufficient to show for some $\beta \in (0, (2-\alpha)\pi^2 )$, 
 there exist $\lambda_L>0$,  such that 
for any $\lambda \in (0, \lambda_L)$, the following holds:
\begin{align}\label{eq-3.6-0613}
\|u_\lambda\|_{p, \beta}< \infty.
\end{align}

Let us first consider the term $D_{2, \lambda}(t,x)$ appeared  in  \eqref{eq-2.3-1}, whose estimate is essentially different from that formulated by M. Foondun and E. Nualart \cite{FoNu} as we will see  below.
Since $D_{2, \lambda}(t,x) =S(\lambda\sigma(u_\lambda(t,x)))$ and then by \eqref{eq-2.16-0601} and \eqref{eq-2.14-0615}, we have that for any $y_0\in [0,1]$
\begin{align}\label{eq-3.7-0613}
\mathbb{E}\Big[\sup_{x\in [0, 1]} |D_{2, \lambda}(t,x)|^p \Big]
\leq & 2^{p-1}\lambda^p K_U^p (\kappa^p \mathbb{E}[B(t)] + \mathbb{E}[|Su(t, y_0)|^p])\\
\leq & 
2^{p-1}\lambda^p K_U^p (\kappa^p  C_1K_1(t) + \mathbb{E}[|Su(t, y_0)|^p]), \notag
\end{align}
where $K_1(t)$ is defined in \eqref{eq-2.18-0612}.
Let us now give the estimate of $K_1(t)$ based on Lemma \ref{lem-2.1}.
In fact, similarly to \eqref{eq-2.18-0601}, by Lemma \ref{lem-2.1}, we can easily see that for all $t \geq 0$ and any $x, y \in [0, 1]$
\begin{align}\label{eq-2.27}
K_1(t)& \leq  \kappa^p(p) C_2 \Big( \int_0^t \int_{0}^1 (t-s)^\alpha |g(t-s, x,z)-g(t-s,y, z)|^{2-\alpha}   \mathbb{E}[ \|u(s)\|_{L^\infty}^p]^{2/p}dzds\Big)^{p/2} \notag \\
&\leq\kappa^p(p) C_3  (\beta^{\frac{\alpha-1}{2}} + ((2-\alpha)\pi^2 -\beta)^{-1})^p  e^{-\beta t} \|u\|_{p, \beta}^p,
\end{align}
where $C_2$ and $C_3$ are generic positive constants and $\alpha \in (0, 1)$ has been used.
%where Lemma \ref{lem-2.1} has  been used for the last inequality. Consequently, by \eqref{eq-2.26} and \eqref{eq-2.27}, we deduce that there exists a constant $C_2(\lambda)= C_2(\lambda; L, p, \beta, \alpha)$ such that 
%\begin{align}\label{eq-2.28}
% \mathbb{E}[B(t)] \leq C_2(\lambda)e^{-\frac{p\beta t}{2}} \|u\|_{p, \beta}.
%\end{align}
On the other hand, it is easier to see that 
\begin{align}\label{eq-3.9-0613}
\mathbb{E}[ |Su(t, y_0)|^p] &\leq 
\kappa^p(p) \left(\int_0^t \int_{0}^1 g^2(t-s,  y_0, y)
\lambda^2 \mathbb{E}[|u_\lambda(s,y)|^p]^{2/p}dyds \right)^{p/2}  \\
& \leq  \kappa^p(p)  \left(\int_0^t \int_{0}^1 g^2(t-s, y_0, y)e^{\beta(t-s)}e^{-\beta(t-s)}
\lambda^2 \mathbb{E}[\|u(s)\|^p]^{2/p}dyds \right)^{p/2} \notag \\
& \leq  \kappa^p(p) e^{-\beta t} \|u\|_{p, \beta}^p  \left(\int_0^t  g^2(2s,  y_0,  y_0)e^{\beta s}ds \right)^{p/2}, \notag
\end{align} 
%where Minkowski's inequality, Burkh\"older inequality and  $\int_0^1 g^2(s, 0, y)dy =g_D(2s,0, 0)$
 %were used from the first inequality to the last one  respectively. Then from \eqref{eq-2.2} and \eqref{eq-2.3}, 
Similarly to Lemma \ref{lem-2.1}, we can easily  show that  if $\beta \in(0, 2\pi^2)$, 
then 
 \begin{align*}
 \sup_{t \geq 0, y\in [0, 1]}\int_0^t  g(2s, y, y)e^{\beta s}ds <C_3(\frac{1}{2\pi^2 -\beta}+ \frac{1}{\sqrt{\beta}}).
 \end{align*}
%{\textcolor{red}{Verify the above estimate!!!!!}}
Thus, combining this with \eqref{eq-3.9-0613}, we see 
\begin{align} \label{eq-2.28-1} 
\mathbb{E}\Big[|Su(t, y_0)|^p] \leq C_4 \kappa^p(p)  (\frac{1}{2\pi^2 -\beta}+ \frac{1}{\sqrt{\beta}}) e^{-\beta t} \|u\|_{p, \beta}^p.  
\end{align} 
Consequently, 
plugging \eqref{eq-2.27} and \eqref{eq-2.28-1} into \eqref{eq-3.7-0613},  
we have that  there exists a constant $C_5$ depending  on  $\beta \in (0, (2-\alpha)\pi^2 )$ such that for all $t\geq 0$  
\begin{align*} %\label{eq-2.29} 
\mathbb{E}\Big[\sup_{x\in [0, 1]} |D_{2, \lambda}(t,x)|^p \Big]\leq C_5 \lambda^p K_U^p  e^{-\beta t} \|u\|_{p, \beta}^p,
\end{align*} 
and then we can get that for each $\beta \in (0, (2-\alpha)\pi^2 )$,
$D_{2, \lambda}(t,\cdot) \in B_{p,\beta}$. 

Let us now turn to consider $D_{1}(t,x)$, which is a easy part. Using \eqref{eq-2.2} and $\eqref{eq-2.3}$, it is easy to see that for any $\beta\in (0, \pi^2)$,
$$
\sup_{t\geq 0, x, y\in [0, 1]}\int_0^1 e^{\beta t} g_D(t,x, y)dy<\infty, 
$$
Since $u_0\in C([0,1])$, the above estimate tells us that $\|D_{1}(t,\cdot)\|_{p, \beta}\leq \|u_0\|_{L^\infty},$ that is,  $D_{1}(t,\cdot) \in B_{p,\beta}$ for $\beta \in(0, p\pi^2)$. 
%\begin{align}
%|D_{1}(t,x)|^p \leq e^{-\frac{\beta t}{2}} \|u_0\|_{L^\infty}^p| \int_0^1 e^{-\frac{\beta t}{2}} g_D(t,x, y)dy|^p.
%\end{align}
%

Consequently, we proved that the solution $u_\lambda(t,x) \in B_{p,\beta}$  for each $\beta \in(0, (2-\alpha)\pi^2)$, which is equivalent to \eqref{eq-3.6-0613}. Therefore, the proof is completed.   
%then using Lemma 2.2, we have that for $\beta \in(0, 2\pi^2)$, there exists a constant $C_6>0$ such that 
%for all $t\geq 0$
%\begin{align}
%\|D_{1}(t)\|_{L^\infty}^p  \leq C_6e^{-\frac{p\beta t}{2}} \|u_0\|_{L^\infty}^p.
%\end{align}
\end{proof}
%%%%%%%%%%%%%%%%%%%%%%%%%%%%%%%%%%%%%%
%%%  Section 3
%%%%%%%%%%%%%%%%%%%%%%%%%%%%%%
\section{Proof of Theorem \ref{thm-3.1-0608}}

The proof of this theorem is essentially depends on the global behavior of the lower bound for the heat kernel $g_D(t,x, y)$. Such estimate is very important and has been studied actively. For large time and short time, we, for example, refer the reader to \cite{Dav-89} and \cite{Zhan-02} respectively. For our aim,  the global behavior  for $g_D(t,x, y)$ is needed, which is studied by many authors for different domains,  please see \cite{Ria-13}, \cite{Son-04} and \cite{Zhan-03}. Based on their research, we will state in the next corollary using our notation and omit its proof.
\begin{lem}\label{lem-3.1-0608}
There exist two strictly positive constants $\kappa_1$ and $\kappa_2$ such that for any $x, y\in [\gamma, 1-\gamma], \gamma \in (0, 1/4)$ 
\begin{align*}
g_D(t,x) \geq \kappa_1\exp(-\pi^2 t) \exp\big(-\kappa_2\frac{|x-y|^2}t \big) (t^{-\frac12} 1_{(0, \gamma^2]}(t) +1_{(\gamma^2, \infty)}(t)), \ t>0.
\end{align*}
\end{lem}
\begin{rem}
In the papers of \cite{Ria-13}, \cite{Son-04} and \cite{Zhan-03}, sharp bounds of both sides of Dirichlet
heat kernel for the Laplacian on bounded  domains with different conditions are discussed. To prove our result, the above lower bound of $g_D(t,x,y)$ is enough. So we do not state the corresponding  upper bound,  see \cite{Ria-13}, \cite{Son-04} and \cite{Zhan-03} and the references therein.
\end{rem}        
Let us now formulate the proof of Theorem \ref{thm-3.1-0608} based on the above corollary.
\begin{proof}
({\it Proof of Theorem \ref{thm-3.1-0608}})
Under our assumptions,  by the positivity of $g_D(t,x,y)$ and Lemma \ref{lem-3.1-0608}, it is easy to see that for all $x \in [\gamma, 1-\gamma]$,
\begin{align}\label{eq-3.3-0608}
D_{1}(t,x)=&\int_0^1 g_D(t,x, y)u_0(y)dy\\
\geq & \inf_{x, y \in [\gamma, 1-\gamma] } g_D(t,x, y) \int_\gamma^{1-\gamma} u_0(y)dy \notag
\\
 \geq &  C_1\inf_{x, y \in [\gamma, 1-\gamma] }\exp(-\pi^2 t) \exp\big(-\kappa_2\frac{|x-y|^2}t \big) (t^{-\frac12} 1_{(0, \gamma^2]}(t) +1_{(\gamma^2, \infty)}(t)) \notag\\
\geq &  C_1\exp(-\pi^2 t) \exp(-\kappa_2t^{-1}) \Big(t^{-\frac12} 1_{(0, \gamma^2]}(t) +1_{(\gamma^2, \infty)}(t) \Big),  \notag
\end{align}
where $C_1 =\kappa_1 \int_\gamma^{1-\gamma} u_0(y)dy>0$ by the assumption $(A1)$. \\
Noting that $\inf_{t\in (0, \kappa^2]}t^{-\frac12\exp(-\kappa_2t^{-1} ) }>0$  and $\exp(-\kappa_2t^{-1} ) \geq \exp(-\kappa_2\gamma^{-2} )$ for $t \geq \gamma^2$, the above estimate \eqref{eq-3.3-0608},  implies that there is a constant $C_2>0$ such that for all $t \geq 0$  and $x \in [\gamma, 1-\gamma]$,
\begin{align}\label{eq-3.4-0608}
\int_0^1 g_D(t,x, y)u_0(y)dy
\geq & C_2 \exp(-\pi^2 t)
\end{align}
By Ito's isometry and the assumption $|\sigma(u)| \geq K_L|u|, u\in \mathbb{R}$, we have that 
\begin{align}\label{eq-3.5-0608}
\mathbb{E}[|D_{2, \lambda}(t,x)|^2] =& \int_{0}^t \int_{0}^1 
g_D^2(t-s,x, y) \mathbb{E}[|\lambda \sigma(u_\lambda(s, y))|^2]dyds  \\
\geq& \lambda^2 K_L^2  \int_{0}^t \int_{0}^1 
g_D^2(t-s,x, y) \mathbb{E}[|u_\lambda(s, y)|^2]dyds. \notag 
\end{align}
From now, let us deal with the term  $$\int_{0}^t \int_{0}^1 
g_D^2(t,x, y) \mathbb{E}[|u_\lambda(s, y)|^2]dyds$$ appeared in the last inequality.
For brevity, let us define  $h(t)=\inf_{x\in [\gamma, 1-\gamma]}\mathbb{E}[|u_\lambda(t,x)|^2]$. Using Lemma  \ref{lem-3.1-0608}, if $x\in [\gamma, 1-\gamma]$, then we have
\begin{align}\label{eq-3.6-0608}
&\int_{0}^t \int_{0}^1 
g_D^2(t,x, y) \mathbb{E}[|u_\lambda(s, y)|^2]dyds \\
\geq & \int_0^t \int_{\gamma}^{1-\gamma} 
%\inf_{x, y \in [\gamma, 1-\gamma]}
 g_D^2(t-s,x, y) h(s)dyds \notag \\
\geq & \int_0^t \int_{\gamma}^{1-\gamma} 
\kappa_1^2\exp(-2\pi^2( t-s)) \exp\big(-2\kappa_2\frac{|x-y|^2}{t-s} \big) \notag \\
& \qquad \times \Big((t-s)^{-\frac12} 1_{(0, \gamma^2]}(t-s) +1_{(\gamma^2, \infty)}(t-s) \Big)^2 h(s)dyds \notag\\
= & \int_0^t \int_{\gamma}^{1-\gamma} 
\kappa_1^2\exp(-2\pi^2( t-s)) \exp\big(-2\kappa_2\frac{|x-y|^2}{t-s} \big) \notag \\
& \qquad \times \Big((t-s)^{-\frac12} 1_{(0, \gamma^2]}(t-s) +1_{(\gamma^2, t)}(t-s) \Big) h(s)dyds \notag%\\
%\geq & \int_{t-\gamma^2}^t   \int_{\gamma}^{1-\gamma} 
%%\kappa_1^2\exp(-2\pi^2 (t-s)) \exp\big(-2\kappa_2\frac{|x-y|^2}{t-s} \big) (t-s)^{-1}  h(s)dyds. \notag\\
%+& \int_0^{t-\gamma^2} \int_{\gamma}^{1-\gamma} 
%\kappa_1^2\exp(-2\pi^2 (t-s)) \exp\Big(-2\kappa_2 \frac{|x-y|^2}{t-s} \Big) h(s)dyds \notag
\end{align} 
Set $A(x; s,t):=[\gamma, 1-\gamma]\cap\{y: |y-x| \geq \sqrt{t-s} \}, s\leq t$ and then noting that $\gamma \in(0,1/4)$,  
we can show that for any $x\in [\gamma, 1-\gamma],$ 
$|A(x;s,t)| \geq \sqrt{t-s}$. Consequently, noting that for $y\in A(x;s,t)$, $|y-x|\leq \sqrt{t-s}$ and the non-increasing of $e^{-x}$,   for $s\in (t-\gamma^2, t)$, we have 
\begin{align} \label{eq-3.7-0608}
& \int_{\gamma}^{1-\gamma} 
\exp(-2\pi^2 (t-s)) \exp\big(-2\kappa_2\frac{|x-y|^2}{t-s} \big) (t-s)^{-1} dy\\
\geq & \int_{A(x;t,s)}
\exp(-2\pi^2 (t-s)) \exp\big(-2\kappa_2\frac{|x-y|^2}{t-s} \big) (t-s)^{-1} dy \notag \\
\geq &  \exp(-2\kappa_2) \exp(-2\pi^2 (t-s)) (t-s)^{-1} |A(x;t,s)| \notag \\
\geq &  \exp(-2\kappa_2) \exp(-2\pi^2 (t-s)) (t-s)^{-\frac12}. \notag
\end{align} 
On the other hand, if  $s\in (0, t-\gamma^2]$, then  for any $x, y \in [\gamma, 1-\gamma]$, $$\exp\big(-2\kappa_2 \frac{|x-y|^2}{t-s} \big) \geq 
\exp\big(-2\kappa_2 \frac{(1-2\gamma)^2}{t-s} \big)
\geq \exp(-2\kappa_2\gamma^{-2})>0.$$
Thus, it is  easy to see that if $s\in (0, t-\gamma^2]$, then  for any $x\in [\gamma, 1-\gamma]$  
\begin{align}\label{eq-3.8-0608}
 & \int_{\gamma}^{1-\gamma} 
\kappa_1^2\exp(-2\pi^2 (t-s)) \exp\big(-2\kappa_2 \frac{|x-y|^2}{t-s} \big)dy \\
\geq & \exp(-2\kappa_2\gamma^{-2})\exp(-2\pi^2 (t-s)). \notag
\end{align}
Inserting \eqref{eq-3.7-0608} and \eqref{eq-3.8-0608} into \eqref{eq-3.6-0608},  we see that there exists a constant $C_3>0$ such that for all $t>0$ and $x\in [\gamma, 1-\gamma]$
\begin{align*}
&\int_{0}^t \int_{0}^1 
g_D^2(t,x, y) \mathbb{E}[|u_\lambda(s, y)|^2]dyds \\
\geq & C_3\int_0^t \exp(-2\pi^2 (t-s)) \Big((t-s)^{-\frac12} 1_{(0, \gamma^2]}(t-s) +1_{(\gamma^2, t]}(t-s) \Big) h(s)ds,
\end{align*} 
and then, noting that $1\geq \gamma (t-s)^{-\frac12}$ for $t-s \geq \gamma^2$, 
we deduce that  
\begin{align*}%\label{eq-3.9-0608}
\int_{0}^t \int_{0}^1 
g_D^2(t,x, y) \mathbb{E}[|u_\lambda(s, y)|^2]dyds  
\geq  C_4\int_0^t \exp(-2\pi^2 (t-s)) (t-s)^{-\frac12}  h(s)ds
\end{align*}
holds for some constant $C_4>0$.\\  
Consequently, combining \eqref{eq-3.5-0608} with 
the above estimate,
%%\eqref{eq-3.9-0608}, 
we have
\begin{align}\label{eq-3.10-0608}
\mathbb{E}[|D_{2, \lambda}(t,x)|^2] \geq C_4\lambda^2 K_L^2 \int_0^t \exp(-2\pi^2 (t-s)) (t-s)^{-\frac12}  h(s)ds.
\end{align}
Noting that  under our assumptions, Ito's isometry gives that  
\begin{align*}
\mathbb{E}[|u_\lambda(t,x)|^2] = \left(\int_0^1 g_D(t,x, y)u_0(y)dy\right)^2 + \mathbb{E}[|D_{2, \lambda}(t,x)|^2].
\end{align*}
Using \eqref{eq-3.4-0608} and \eqref{eq-3.10-0608}, we have that for any $x\in[\gamma, 1-\gamma]$,
\begin{align*}
\mathbb{E}[|u_\lambda(t,x)|^2] \geq C_2^2 \exp(-2\pi^2 t)+
C_4\lambda^2 K_L^2 \int_0^t \exp(-2\pi^2 (t-s)) (t-s)^{-\frac12}  h(s)ds,
\end{align*}
which implies that for any $t>0$  
\begin{align}\label{eq-3.11-0608}
h(t) \geq C_2^2 \exp(-2\pi^2 t)+
C_4\lambda^2 K_L^2 \int_0^t \exp(-2\pi^2 (t-s)) (t-s)^{-\frac12}  h(s)ds.
\end{align}
Let us now define $H(t)= \exp(2\pi^2 t)h(t)$ and then from \eqref{eq-3.11-0608}, we get the following relation for $H(t)$:
for any $t>0$  
\begin{align*}
H(t) \geq C_2^2 +
C_4\lambda^2 K_L^2 \int_0^t (t-s)^{-\frac12} H(s)dyds.
\end{align*} 
Finally, owing to Gronwall's inequality, we can easily obtain that 
\begin{align*}
H(t) \geq C_2^2 \exp(C_4^2\lambda^4 K_L^4 t), \ t>0,
\end{align*} 
which completes our proof.
\end{proof}

\section{Proof of Theorem \ref{thm-1.3-0608} and Corollary \ref{cor-1.5-0611} }
The first aim of this part is to formulate the proof of Theorem \ref{thm-1.3-0608} based on Theorem \ref{thm-3.1-0608} and the second one is to prove  Corollary \ref{cor-1.5-0611}  as the application of Theorem \ref{thm-1.3-0608}.  %To prove, it, the following stronger result will be proved.

\begin{proof}{(\it Proof of Theorem \ref{thm-1.3-0608})}
% The first part of this theorem is trivial by Theorem \ref{thm-1.2-0608}, because, 
% for any $p\geq 2$ and $x\in [0, 1]$, $$\mathbb{E} [|u_\lambda(t,x)|^p] \leq \mathbb{E}[\|u_\lambda(t)\|_{L^\infty}^p]. $$ 
Owing to Theorem \ref{thm-1.1-0531} and $|u_\lambda(t,x)| \leq \|u_\lambda\|_{L^\infty}$, it is sufficient for us  to verify  the lower bound, i.e.,  for any $\lambda \in (\lambda_U, \infty)$ and $ x\in [\gamma, 1-\gamma]$
\begin{align}\label{eq-4.1-0614}
 0 <\liminf_{t\to  \infty}\frac1t\log\mathbb{E}[|u_\lambda(t,x)|^p]. 
\end{align}
%As we explained at the beginning of this section, we will only state the proof of $(ii)$. 
%This can easily proved by Theorem \ref{thm-1.3-0608}. In fact,  
Jensen's inequality tells us that for any $p>2$,
$$\mathbb{E}[|u_\lambda(t,x)|^2]^{1/2} \leq \mathbb{E}[|u_\lambda(t,x)|^p]^{1/p},$$
and thus  it is enough for us to show that  there  exists a large enough $\lambda_U$, \eqref{eq-4.1-0614} is satisfied when $p=2$.
 However, this is a quick result of Theorem \ref{thm-3.1-0608}. In fact,  since \eqref{eq-3.1-0608} holds for any $t>0$, we easily know that 
\begin{align*}
 \log \Big(\inf_{x\in [\gamma, 1-\gamma]}\mathbb{E}[|u_\lambda(t,x)|^2]\Big) \geq \log \kappa_1 + (\kappa_2 \lambda^4 K_L^4 -2 \pi^2)t,\  t\geq 0.
\end{align*}
Dividing both sides of the above inequality by $t$ and taking the infimum limit as $t \to \infty$, we see that
\begin{align*}
 \liminf_{t\to \infty}\frac{1}{t}\log  \Big( \inf_{x\in [\gamma, 1-\gamma]}\mathbb{E}[|u_{\lambda}(t,x)|^2] \Big) \geq \kappa_2 \lambda^4 K_L^4 -2\pi^2.
\end{align*}
Let us take  $\lambda_U=\Big( \frac{2\pi^2}{\kappa_2 K_L^4 }\Big)^{1/4}$. Then, from the above inequality, we have that for all $\lambda > \lambda_U$,   \eqref{eq-4.1-0614} holds for $p=2$.
\end{proof}

\begin{proof}{\it (Proof of Corollary \ref{cor-1.5-0611}) }
The first part comes immediately from Theorem \ref{thm-1.2-0608} by noting that for any $p\geq 2$, $\mathbb{E}[\|u_\lambda(t)\|_{L^p}^p]\leq \mathbb{E}[\|u_\lambda(t)\|_{L^\infty}^p]$.\\
Let us now consider the proof of the second part.
%Let us begin with the proof of the lower bound, which is easily deduced from Theorem  \ref{thm-3.1-0608}.
By Fubini's theorem and Jensen's inequality, for any $p>2$,
\begin{align*}
\mathbb{E}\left[\int_0^1 |u_\lambda(t,x)|^p dx \right] 
 \geq & \int_0^1\mathbb{E}[u_\lambda^2(t,x)]^{p/2}dx \notag \\
 \geq & \int_\gamma^{1-\gamma}\inf_{x\in [\gamma, 1-\gamma]}\mathbb{E}[u_\lambda^2(t,x)]^{p/2}dx 
 \notag \\
=  & (1-2\gamma)\inf_{x\in [\gamma, 1-\gamma]}\mathbb{E}[u_\lambda^2(t,x)]^{p/2}. \notag 
\end{align*}
 Combining the above estimate with Theorem  \ref{thm-3.1-0608},  we deduce that for all $t>0$
\begin{align}\label{eq-1.15.-0611}
\log\mathcal{E}_p(t, \lambda) \geq &
\frac{1}{p}\log\left((1-2\gamma)\inf_{x\in [\gamma, 1-\gamma]}\mathbb{E}[u_\lambda^2(t,x)])^{p/2}\right)\\
 \geq & \frac{1}{p} \log\big((1-2\gamma)\kappa_1^{p/2} \big) +\frac12(\kappa_2 \lambda^4 K_L^4 -2\pi^2)t. \notag
\end{align}
Finally, dividing both sides by $t$ and taking the infimum limit  as $t \to \infty$, we see that the lower bound holds for any $\lambda > \lambda_U$.  
\end{proof}

%%%%%%%%%%%%%%%%%%%%%%%%%%%%%%%%%%
%%% Section 6
\section{Proof of Theorem \ref{thm-1.6-0610} and Corollary \ref{cor-1.7-0613} }

\begin{proof}{\it (Proof of Theorem \ref{thm-1.6-0610})}
It is first task to proof the upper bound.  It is easer than the proofs of Theorem \ref{thm-1.1-0531} and Theorem \ref{thm-1.2-0608}.
By the assumption $(A_1)$ on $u_0$ and the estimate \eqref{eq-2.2} of $g_D(t,x,y)$, it is easy to know that there exists a constant $C_1$ such that for all $t>0$, $$\sup_{x\in [0, 1]}D_{1}(t,x) \leq C_1.$$ %recalling that $u_0\in C([0, 1])$.

By a similar  approach  used in the proof of Theorem \ref{thm-1.1-0531}, we can prove  that there is a constant $C_2>0$ such that  for each $t>0$ and any $x\in [0, 1]$,
\begin{align*}
\mathbb{E}[|D_{2, \lambda}(t,x)|^p]^{2/p} \leq & C_2\lambda^2 K_U^2 \int_0^tg_D^2(t-s,x, y) \mathbb{E}[|u_\lambda(s, y)|^p]^{2/p}dyds   \\
\leq & C_2\lambda^2 K_U^2 \int_0^tg_D^2(t-s,x, y) \sup_{y\in [0, 1]}\mathbb{E}[|u_\lambda(s, y)|^p]^{2/p}dyds \notag \\
\leq & C_2\lambda^2 K_U^2 \int_0^tg_D(2(t-s),x,x)) \sup_{y\in [0, 1]}\mathbb{E}[|u_\lambda(s, y)|^p]^{2/p}ds \notag
\\
\leq & C_2\lambda^2 K_U^2 \int_0^t\frac{1}{\sqrt{4\pi(t-s)}}\sup_{y\in [0, 1]}\mathbb{E}[|u_\lambda(s, y)|^p]^{2/p}ds. \notag
\end{align*}
Hence, by Minkowski's inequality and the above estimates, there exist constant $C_3$ and $C_4$ such that for  any $t>0$
\begin{align*}
\sup_{x\in [0,1]} \mathbb{E}[|u_{\lambda}(t, x)|^p]^{2/p}
\leq C_3+ C_4\lambda^2 K_U^2 \int_0^t\frac{1}{\sqrt{(t-s)}}\sup_{y\in [0, 1]}\mathbb{E}[|u_\lambda(s, y)|^p]^{2/p}ds,
\end{align*}
and then for  any $t>0$
\begin{align*}
\sup_{x\in [0,1]} \mathbb{E}[|u_{\lambda}(t, x)|^p]^{2/p}
\leq C_3+2C_3C_4 \sqrt{t}+ C_4^2\lambda^4 K_U^4 \int_0^t\mathbb{E}[|u_\lambda(s, y)|^p]^{2/p}ds,
\end{align*}
As a consequence, owing to  Gronwall inequality,
we have for all 
$t>0$
\begin{align*}
\sup_{x\in [0,1]} \mathbb{E}[|u_{\lambda}(t, x)|^p]^{2/p}
\leq (C_3+2C_3C_4 \sqrt{t})\exp(C_4^2\lambda^4 K_U^4 t),
\end{align*}
which immediately implies the upper bound.

Let us now begin to state the proof of the lower bound, i.e., for some $c_p>0$, 
$$
c_p  K_{L}^4 t\leq   \liminf_{\lambda \to \infty} \lambda^{-4} K_L\log  \Big(\inf_{x\in [\gamma, 1-\gamma]}\mathbb{E}[|u_\lambda(t,x)|^p]  \Big).
$$
It is also a direct conclusion of Theorem  \ref{thm-3.1-0608}. From Jensen's inequality,  it is sufficient for us to show that the above lower bound holds for $p=2$. By Theorem  \ref{thm-3.1-0608}, for any $t>0$
\begin{align*}
\log\left(\inf_{x\in [\gamma, 1-\gamma]}\mathbb{E}[u_\lambda^2(t,x)] \right) \geq \log\kappa_1 +(\kappa_2 \lambda^4 K_L^4 -2\pi^2)t.
\end{align*}
Let us first divide both 
sides of the above inequality by $\lambda^4$ and then take the infimum limit as $\lambda \to \infty$,
the lower bound can be proved immediately.
\end{proof}
In the end, let us state the proof of Corollary \ref{cor-1.7-0613} as the application of Theorem \ref{thm-1.6-0610}. 
%%%%%%%%%%%%%%%%%%%%%%%%%%%%%%%
%%%%
\begin{proof}{\it (Proof of Corollary \ref{cor-1.7-0613})}
The lower bound can be easily deduced from \eqref{eq-1.15.-0611}, appeared in the proof of Corollary \ref{cor-1.5-0611}. In fact, dividing both 
sides of \eqref{eq-1.15.-0611} by $\lambda^4$ and then taking the infimum limit as $\lambda \to \infty$, we can easily obtain the lower bound for each $t>0$.
%By Fubini's theorem and Jensen's inequality, for any $p>2$,
%\begin{align}
%\log\mathcal{E}_{p, t} (t, \lambda)
%=&\frac{1}{p}\log\left(\mathbb{E}[\int_0^1 |%u_\lambda(t,x)|^p dx]\right) 
%\\
% \geq & \frac{1}{p} \log\left(\int_0^1\mathbb{E}[u_\lambda^2(t,x)]^{p/2}dx \right) \notag \\
% \geq & \frac{1}{p} \log\left(\int_\gamma^{1-\gamma}(\inf_{x\in [\gamma, 1-\gamma]}\mathbb{E}[u_\lambda^2(t,x)])^{p/2}dx \right)
% \notag \\
%=  & \frac{1}{p} \log\left((1-2\gamma)\inf_{x\in [\gamma, 1-\gamma]}\mathbb{E}[u_\lambda^2(t,x)])^{p/2} \right). \notag 
%\end{align}
%Using Theorem  \eqref{thm-3.1-0608}, we have that 
%\begin{align*}
%\log\left((1-2\gamma)\inf_{x\in [\gamma, 1-\gamma]}\mathbb{E}[u_\lambda^2(t,x)])^{p/2}\right) \geq \log((1-2\gamma)\kappa_1^{p/2}) +p(\kappa_2 \lambda^4 K_L^4 -2\pi^2)t/2,
%\end{align*}
%which gives immediately that 
%\begin{align*}
%\liminf_{\lambda \to \infty}\lambda^4 \log\left((1-2\gamma)\inf_{x\in [\gamma, 1-\gamma]}\mathbb{E}[u_\lambda^2(t,x)])^{p/2} \right) \geq \kappa_2 \lambda^4 K_L^4t.
%\end{align*}
%Consequently, we can conclude the proof of the lower bound.

On the other hand, noting that $$\mathbb{E}[\|u(t)\|_{L^p}^p] \leq \sup_{x\in [0, 1]}\mathbb{E}[|u_\lambda(t,x)|^p],$$ the upper bound is obtained immediately by Theorem \ref{thm-1.6-0610}. As a result, the proof of this corollary is completed.
\end{proof}

\vskip 1cm

\section{Appendix } 
According to \cite{GRR} and \cite{Wal}, let us rewrite the celebrated Garsia-Rodemich-Rumsey theorem for our purpose.
Let $\Phi:\mathbb{R} \to [0, \infty)$ be a Young function (a convex and  even  function with $\Phi(0)=0$ and $\lim_{x\to \infty} \Phi(x)=\infty$) and let $\phi: [0,1] \to [0, \infty)$ be continuous and increasing with $\phi(0)=0$. 
\begin{thm}\label{GRR}
Let $\Phi$ and $\phi$ be defined as above. If $f$ is a measurable function on $[0, 1]$ such that 
\begin{align*}
\int_0^1\int_0^1 \Phi\Big(\frac{f(x)-f(y)}{\phi(|x-y|) } \Big)dxdy=B<\infty,
\end{align*}
then 
\begin{align*}
|f(x)-f(y)| \leq 8 \int_0^{|x-y|} \Phi^{-1}\Big(\frac{B}{u} \Big)dp(u)\  a.e. \ x, y\in [0,1].
\end{align*}
\end{thm}

\begin{center}{\bf Acknowledgements} \end{center}
{\it The author was supported  in part by
%The  author is  grateful to the support of 
Grant-in-Aid for
Young Scientist (B) 25800060 and Grant-in-Aid for Scientific Research (C) 24540198 from Japan Society for the Promotion of Science(JSPS).
%  Isaac Newton Institute for Mathematical Sciences for their hospitality and support during the program ``Stochastic Partial Differential Equations (SPDEs)", where part of this work was done.
 }

\bibliographystyle{plain}

\end{document}